\documentclass[12pt]{amsart}
\usepackage{amsmath, amsthm, amsfonts, amssymb}
\usepackage{enumerate}
\usepackage[bookmarks=false]{hyperref}
\usepackage{geometry}

\theoremstyle{plain}
\newtheorem{thm}{Theorem}[section]
\newtheorem*{theo}{Theorem}
\newtheorem*{thma}{Theorem A}
\newtheorem*{thmb}{Theorem B}
\newtheorem*{thmc}{Theorem C}

\newtheorem{lem}[thm]{Lemma}
\newtheorem{prop}[thm]{Proposition}
\newtheorem{cor}[thm]{Corollary}

\newtheorem*{implication}{Implication}

\theoremstyle{definition}
\newtheorem{defn}[thm]{Definition}
\newtheorem*{defn*}{Definition}

\newtheorem{rem}[thm]{Remark}

\newtheorem*{notation}{Notation}

\numberwithin{equation}{section}

\newcommand{\N}{\mathbb{N}}
\newcommand{\R}{\mathbb{R}}
\newcommand{\C}{\mathbb{C}}
\newcommand{\Z}{\mathbb{Z}}

\newcommand{\conv}{\operatorname{conv}}
\newcommand{\spa}{\operatorname{span}}

\newcommand{\ootimes}{\, \overline{\otimes} \,}

\newcommand{\id}{\operatorname{id}}

\newcommand{\PSL}{\operatorname{PSL}}

\begin{document}

\title[W$^*$-superrigidity of Gaussian actions]{W$^*$-superrigidity of mixing Gaussian actions of rigid groups}

\begin{abstract}
We generalize W$^*$-superrigidity results about Bernoulli actions of rigid groups to general mixing Gaussian actions. We thus obtain the following: If $\Gamma$ is any ICC group which is w-rigid (i.e.\ it contains an infinite normal subgroup with the relative property (T)) then any mixing Gaussian action $\Gamma \curvearrowright X$ is W$^*$-superrigid. More precisely, if $\Lambda \curvearrowright Y$ is another free ergodic action such that the crossed-product von Neumann algebras are isomorphic $L^\infty(X) \rtimes \Gamma \simeq L^\infty(Y) \rtimes \Lambda$, then the actions are conjugate. We prove a similar statement whenever $\Gamma$ is a non-amenable ICC product of two infinite groups. 
\end{abstract}

\author{R\'emi Boutonnet}

\address{ENS Lyon \\
UMPA UMR 5669 \\
69364 Lyon cedex 7 \\
France}

\email{remi.boutonnet@ens-lyon.fr}


\maketitle

\section{Introduction}

Most known examples of finite von Neumann algebras are constructed from discrete groups or equivalence relations. Thus the question of understanding which data of the initial group or equivalence relation is remembered in the construction of the associated von Neumann algebra is fundamental if one wants to classify finite von Neumann algebras.  
This problem is usually very hard, but a dramatic progress has been made possible in the last decade thanks to Sorin Popa's deformation/rigidity theory (see \cite{Po07b,Ga10,Va10a} for surveys).


The first W$^*$-rigidity result in the framework of group-measure space constructions is Popa's strong rigidity theorem \cite{Po06a,Po06b}. Assume that $\Gamma \curvearrowright X = X_0^\Gamma$ is a Bernoulli action and that $\Lambda \curvearrowright Y$ is a probability measure preserving (pmp) free ergodic action of an ICC w-rigid group (i.e.\ which contains an infinite normal subgroup with the relative property (T), \cite{Po06a}). Popa shows in \cite{Po06b} that if the crossed-product von Neumann algebras of these actions are isomorphic, then the actions are conjugate.
This is the first result that deduces conjugacy of two actions out of an isomorphism of their crossed product von Neumann algebra.

Later on, Ioana managed to prove the following very general W$^*$-superrigidity result about Bernoulli shifts, which is a natural continuation to Popa's strong rigidity result. For more historical information and results on $W^*$-superrigidity, see for instance \cite{Pe09,PV09,HPV10,Io11,IPV11} and the introductions therein.

\begin{theo}[Ioana, \cite{Io11}]
Let $\Gamma$ be an ICC w-rigid group. Then the Bernoulli action $\Gamma \curvearrowright^\sigma [0,1]^\Gamma$ is W$^*$-superrigid:

Let $\rho$ be any free ergodic measure-preserving action of any group $\Lambda$. Denote by $M$ and $N$ the crossed-product von Neumann algebras associated with $\sigma$ and $\rho$ respectively, and assume that $M \simeq N$. Then $\Lambda$ is isomorphic to $\Gamma$, and the actions $\sigma$ and 
$\rho$ are conjugate.
\end{theo}

The rigidity of the von Neumann algebra in the theorem above comes from a tension between the property (T) of the group $\Gamma$ and the deformability of Bernoulli actions. This tension is exploited via Popa's deformation/rigidity strategy.
Using a similar strategy of proof, Ioana, Popa and Vaes later proved the W$^*$-superrigidity of Bernoulli actions for other groups, relying this time on the spectral gap type rigidity discovered by Popa in \cite{Po08}.

\begin{theo}[Ioana-Popa-Vaes, \cite{IPV11}]
Let $\Gamma$ be a non-amenable ICC group which is the product of two infinite groups $\Gamma = \Gamma_1 \times \Gamma_2$. Then the Bernoulli action $\Gamma \curvearrowright^\sigma [0,1]^\Gamma$ of $\Gamma$ is W$^*$-superrigid.
\end{theo}

Both of the proofs of Ioana's theorem and Ioana-Popa-Vaes' theorem seemed to deeply rely on the very particular structure of Bernoulli actions. We will show that this is not the case, and generalize these results to {\it Gaussian actions}.

Let $\Gamma$ be a countable group and $\pi: \Gamma \rightarrow \mathcal{O}(H)$ an orthogonal representation of $\Gamma$ on a real Hilbert space $H$. Recall that there exist (see \cite{PS10} for instance) a standard probability space $(X,\mu)$ and a pmp action of $\Gamma$ on $X$, such that $H \subset L^2(X)$, as representations of $\Gamma$. This action is called the {\it Gaussian action} induced by the orthogonal representation $\pi$.

We generalize Ioana's result as follows.

\begin{thma}
Let $\Gamma$ be an ICC w-rigid group and $\pi: \Gamma \rightarrow \mathcal{O}(H)$ any mixing orthogonal representation of $\Gamma$.
Then the Gaussian action $\sigma_\pi$ associated with $\pi$ is W$^*$-superrigid.
\end{thma}

However, in order to apply Popa's spectral gap argument, one has to make an extra assumption on the initial representation $\pi$. Ioana-Popa-Vaes' theorem then becomes a particular case of the following result.

\begin{thmb}
Let $\Gamma$ be a non-amenable ICC group which is the product of two infinite groups, and consider a mixing orthogonal representation $\pi: \Gamma \rightarrow \mathcal{O}(H)$ of $\Gamma$. Assume that some tensor power of $\pi$ is weakly contained in the regular representation.
Then the Gaussian action $\sigma_\pi$ associated with $\pi$ is W$^*$-superrigid.
\end{thmb}

To prove Theorem A and Theorem B, we will adapt the proof used by Ioana and Ioana-Popa-Vaes to the context of Gaussian action. Let us recall the general strategy of their proof.

\subsection*{Steps of the proof in the Bernoulli case.}

Let $\Gamma$ be a group as in theorem A or B and $\Gamma \curvearrowright X = [0,1]^\Gamma$ the corresponding Bernoulli action. Assume that $\Lambda \curvearrowright (Y,\nu)$ is another pmp, free ergodic action such that \[L^\infty(X)\rtimes \Gamma = L^\infty(Y) \rtimes \Lambda.\]
Put $A = L^\infty(X)$, $B = L^\infty(Y)$ and $M = A \rtimes \Gamma$.

Thanks to Popa's orbit equivalence superrigidity theorems \cite[5.2 and 5.6]{IPV11} and \cite[Theorem 1.3]{Po08}, one only has to show that the two actions are orbit equivalent. More concretely it is enough to prove, by a result of Feldman and Moore \cite{FM77}, that $B$ is unitaly conjugate to $A$ inside $M$.

The main idea of the proof, due to Ioana, is to exploit the information given by the isomorphism $M = B \rtimes \Lambda$ via the dual co-action
\begin{eqnarray*}
\Delta: M & \rightarrow & M \ootimes M \\
bv_s & \mapsto & bv_s \otimes v_s,
\end{eqnarray*}
$b \in B$, $s \in \Lambda$ ($v_s$, $s \in \Lambda$, denote the canonical unitaries corresponding to the action of $\Lambda$).
This $*$-homomorphism $\Delta$ allows us to play against each other two data of the single action $\Gamma \curvearrowright X$: the rigidity of $\Delta(L\Gamma)$, and the malleability of the algebra $M \ootimes M = (A \ootimes A) \rtimes (\Gamma \times \Gamma)$.

Assume that $B$ is not unitarily conjugate to $A$, or equivalently that $B \nprec_M A$ by \cite[Theorem A.1]{Po06c}. We refer to Section 2.1 for the definition of Popa's intertwining symbol ``$\prec$". The rest of the proof can be divided into four steps, which lead to a contradiction.

\begin{enumerate}[{\bf Step (1).}] 
\item One shows that there exists a unitary $u \in M \ootimes M$ such that \[u \Delta(L\Gamma) u^* \subset L\Gamma \ootimes L\Gamma.\]
\item One further proves that the algebra $C := \Delta(A)' \cap (M \ootimes M)$ satisfies \[C \prec_{M \ootimes M} A \ootimes A.\]
\item The previous steps, and an enhanced version of Popa's conjugacy criterion \cite[Theorem 5.2]{Po06b} imply that roughly there exist a unitary $v \in M \ootimes M$, a group homomorphism $\delta : \Gamma \rightarrow \Gamma \times \Gamma$, and a character $\omega:\Gamma \rightarrow \C$ such that \[v\Delta(C)v^* = A \ootimes A \text{ and } v\Delta(u_g)v^* = \omega(g) u_{\delta(g)}, \, \forall g \in \Gamma.\]
\item Using Step (3), one can now show that if a sequence $(x_n)$ in $M$ has Fourrier coefficient (with respect to the decomposition $M = A \rtimes \Gamma$) which tend to zero pointwise in norm $\Vert \cdot \Vert_2$, then this is also the case of the sequence $\Delta(x_n)$, with respect to the decomposition $M \ootimes M = (M \ootimes A) \rtimes \Gamma$. This easily contradicts the fact that $B \nprec_M A$.
\end{enumerate}

\subsection*{What has to be adapted}

First note that Popa's orbit equivalence superrigidity results (\cite[5.2 and 5.6]{IPV11} and \cite[Theorem 1.3]{Po08}) are still valid for Gaussian actions as in Theorem A or Theorem B. Thus we only have to prove Steps (1)-(4) for such Gaussian actions.

Steps (3) and (4) are very general, and will work for any mixing action satisfying the conclusions of steps (1) and (2).

Step (1) is the result of Popa's deformation/rigidity strategy so it should not be specific to Bernoulli shifts. In \cite{IPV11}, it was a direct consequence of \cite[Corollary 4.3]{IPV11}. Using the results in \cite{Bo12}, one can easily get the Gaussian counterpart of \cite[Corollary 4.3]{IPV11}, namely Corollary \ref{corIPV}.

Finally, Step (2) relies on a beautiful localization theorem due to Ioana, \cite[Theorem 6.1]{Io11}. That theorem states that if $D$ is an abelian subalgebra of $M \ootimes M$ which is normalized by ``enough" unitaries in $L(\Gamma \times \Gamma)$, then either $D' \cap (M \ootimes M) \prec A \ootimes A$, or $D \prec M \ootimes L\Gamma$, or $D \prec L\Gamma \ootimes M$. 
This theorem is applied to $D = \Delta(A)$. Using mixing properties, and the fact that $B \nprec A$, the last two cases cannot hold and Step (2) follows.

{\it So the point of the whole proof of theorems A and B is to generalize Ioana's localization theorem \cite[Theorem 6.1]{Io11}}. It will be done in Section 3, Theorem \ref{keythm}. We explain below the main difficulties to obtain such a generalization.

\subsection*{Main difficulties in the generalization}
Unlike Bernoulli shifts, general mixing Gaussian actions do not satisfy the following properties, which were crucial in Ioana's argument.
\begin{itemize}
\item {\it Cylinder structure}: If $\Gamma \curvearrowright [0,1]^\Gamma$ is a Bernoulli action, we call {\it finite cylinder subalgebra} a subalgebra of $A = L^\infty([0,1]^\Gamma)$ of the form $A_F = L^\infty([0,1]^F)$, for some finite subset $F \subset \Gamma$. Then the union of all finite cylinder subalgebras is a strongly dense $*$-subalgebra $A_0$ of $A$, which is stable under the action of $\Gamma$. In fact, $A_0$ is a graded $\C \Gamma$-module;
\item {\it ``strong compactness" property}: If $\Gamma \curvearrowright [0,1]^\Gamma$ is a Bernoulli action, there exist a strongly dense $*$-subalgebra $A_0$ of $L^\infty([0,1]^\Gamma)$ such that for any $a,b \in A_0 \ominus \C$, $\langle \sigma_g(a),b \rangle = 0$ if $g \in \Gamma$ large enough.
\end{itemize}

The use of the strong compactness property can be avoided using $\varepsilon$-orthogonality and a trick involving convex combinations. To avoid using the cylinder structure, the idea is to replace cylinders by general finite dimensional subsets of $L^\infty(X)$, and use a multiple mixing property automatically enjoyed by mixing Gaussian actions.

\begin{defn*}
A trace-preserving action $\Gamma \curvearrowright^\sigma A$ of a countable group on an abelian von Neumann algebra is {\it $2$-mixing} if for any $a, b, c \in A$, the quantity $\tau(a\sigma_{g}(b)\sigma_{h}(c))$ tends to $\tau(a)\tau(b)\tau(c)$ as $g, h, g^{-1}h$ tend to infinity.
\end{defn*}

In fact, each step of the proof still holds for general $s$-malleable actions (in the sense of Popa \cite{Po08}) which are $2$-mixing:

\begin{thmc}
Let $\Gamma$ be an ICC group and $\Gamma \curvearrowright^\sigma (X,\mu)$ be a free ergodic action of $\Gamma$. Assume that $\sigma$ is $2$-mixing and $s$-malleable in the sense of Popa \cite{Po08} and that one of the following two conditions holds.
\begin{itemize}
\item $\Gamma$ is w-rigid or
\item $\Gamma$ is non-amenable and is isomorphic to the product of two infinite groups, and some tensor power of the Koopman representation $\sigma_{\vert L^2(X) \ominus \C}$ is weakly contained in the regular representation of $\Gamma$.
\end{itemize}
Then $\sigma$ is W$^*$-superrigid.
\end{thmc}

\subsection*{Organization of the article}

Apart from the introduction, this article contains three other sections. Section 2 is the preliminary section, in which we recall Popa's intertwining techniques, and several facts on Gaussian actions that we proved in \cite{Bo12}. Section 3 is devoted to proving Theorem \ref{keythm}, which generalizes \cite[Theorem 6.1]{Io11}. Finally, an extra-section is devoted to constructing a large class of II$_1$-factors which are not isomorphic to group von Neumann algebras. This will be a direct application of Theorem \ref{keythm} and the work of Ioana-Popa-Vaes.

\subsection*{Acknowledgement}

We warmly thank Cyril Houdayer and Adrian Ioana for their valuable advice and comments about this work.

\section{Preliminaries}

\subsection{Intertwining by bimodules}

We recall here an essential tool introduced by Popa, the so-called intertwining by bimodules' theorem.

\begin{thm}[Popa, \cite{{Po06a},{Po06d}}]
\label{intertwining}
Let $P,Q \subset M$ be finite von Neumann algebras (with possibly non-unital inclusions). Then the following are equivalent.
\begin{itemize}
\item There exist projections $p \in P$, $q \in Q$, a normal $*$-homomorphism $\psi: pPp \rightarrow qQq$, and a non-zero partial isometry $v \in pMq$ such that $xv = v\psi(x)$, for all $x \in pPp$;
\item There exists a $P$-$Q$ subbimodule $H$ of $L^2(1_PM1_Q)$ which has finite index when regarded as a right $Q$-module;
\item There is no net of unitaries $(u_i) \in \mathcal U(P)$ such that for all $x,y \in M$, \[\Vert E_Q(1_Qx^*u_iy1_Q) \Vert_2 \rightarrow 0.\]
\end{itemize}
\end{thm}

Following \cite{Po06a}, if $P,Q \subset M$ satisfy these conditions, we say that {\it a corner of $P$ embeds into $Q$ inside $M$}, and we write $P \prec_M Q$.

Assume that we are in the concrete situation where $M$ is of the form $M = B \rtimes \Gamma$ for some trace preserving action of $\Gamma$ on a finite von Neumann algebra and $Q = B$. Denote by $u_g, g \in \Gamma$ the canonical unitaries in $M$ implementing the action of $\Gamma$.
Then it is easy to check that a subalgebra $P \subset M$ satisfies $P \nprec B$ if and only if there exists a net of unitaries $v_i \in \mathcal U(P)$ such that \[ \Vert E_B(v_iu_g^*) \Vert_2 \rightarrow 0, \, \forall g \in \Gamma.\]
This result can be improved as follows.

\begin{lem}[Ioana, \cite{Io11}, Theorem 1.3.2]
\label{ioanascriterion}
Let $\Gamma \curvearrowright B$ be a trace preserving action on a finite von Neumann algebra $(B,\tau)$. Put $M = B \rtimes \Gamma$, and let $P \subset M$ be a von Neumann subalgebra. Then $P \nprec B$ if and only if there exists a net of unitaries $v_i \in \mathcal U(P)$ such that \[ \lim_n \left( \sup_{g \in \Gamma} \Vert E_B(v_iu_g^*) \Vert_2 \right) = 0.\]
\end{lem}

Another natural question one may wonder: What does it mean to embed into the group algebra $L\Gamma$ inside a crossed-product algebra $M = A \rtimes \Gamma$? In some specific circumstances, this implies the unitary conjugacy into $L\Gamma$, as the following standard result shows.

We denote by $\mathcal{N}_M(Q) = \lbrace u \in \mathcal{U}(M), \, uQu^* = Q \rbrace$ the {\it normalizer} of a subalgebra $Q$ of a von Neumann algebra $M$. The {\it quasi-normalizer} $\mathcal{QN}_M(Q)$ of $Q$ in $M$ is the *-subalgebra of $M$ formed by $Q$-$Q$ finite elements. We recall that an element $x \in M$ is $Q$-$Q$ finite if there exist $x_1, \cdots, x_k \in M$ such that \[xQ \subset \sum_{i=1}^k Qx_i \text{ and } Qx \subset \sum_{i=1}^k x_iQ.\]

\begin{prop}
\label{embedgroup}
Let $\Gamma \curvearrowright A$ be a free mixing action of an ICC group $\Gamma$ on an abelian von Neumann algebra, and let $N$ be a type II$_1$ factor. Put $M = (A \rtimes \Gamma) \ootimes N$, and assume that $Q \subset pMp$ is a von Neumann subalgebra such that $Q \nprec_M 1 \otimes N$. Put $P = \mathcal{QN}_{pMp}(Q)''$.
\begin{enumerate}
\item If $Q \prec_M L\Gamma \ootimes N$ then there exists a non-zero partial isometry $v \in pM$ such that $vv^* \in \mathcal Z(P)$ and $v^*Pv \subset L\Gamma \ootimes N$.
\item If $rQ \prec_M L\Gamma \ootimes N$ for all $r \in Q' \cap pMp$ then there exists a unitary $u \in \mathcal U(M)$ such that $uPu^* \subset L\Gamma \ootimes N$.
\end{enumerate}
\end{prop}
\begin{proof}
(1) By assumption, there exist projections $p_0 \in Q$, $q \in L\Gamma \ootimes N$, a non-zero partial isometry $v \in p_0Mq$ and a *-homomorphism $\varphi : p_0Qp_0 \rightarrow q(L\Gamma \ootimes N)q$ such that for all $x \in p_0Qp_0$, one has $xv = v\varphi(x)$.

By \cite[Remark 3.8]{Va08}, one can assume that $\varphi(p_0Qp_0) \nprec_M 1 \otimes N$. Hence \cite[Theorem 3.1]{Po06a} implies that $\mathcal{QN}_{qMq}(\varphi(p_0Qp_0))'' \subset L\Gamma \ootimes N$. But we see that $v^*Pv \subset \mathcal{QN}_{qMq}(\varphi(p_0Qp_0))''$. Moreover $vv^* \in p_0(Q' \cap M) \subset P$. However $vv^*$ is not necessarily in $\mathcal Z(P)$ but one can modify $v$ as follows to obtain such a condition.

Take partial isometries $v_1,\cdots,v_k \in P$ such that $v_i^*v_i \leq vv^*$, $i = 1,\cdots,k$ and $\sum_{i=1}^k v_iv_i^*$ is a central projection in $P$. Since $L\Gamma \ootimes N$ is a factor, there exist partial isometries $w_1,\cdots,w_k \in L\Gamma \ootimes N$ such that $w_iw_i^* = v^*v_i^*v_iv$ and $w_iw_j^* = 0$, for all $1 \leq i\neq j \leq k$. Define a non-zero partial isometry by $w = \sum_i v_ivw_i \in pM$. We get
\begin{itemize}
\item $ww^* = \sum_i v_ivw_iw_i^*v^*v_i^* = \sum_i v_iv_i^* \in \mathcal{Z}(P)$;
\item $w^*Pw \subset \sum_i w_i^*v^*Pvw_i \subset L\Gamma \ootimes N$.
\end{itemize}
(2) Consider a maximal projection $r_0 \in Q' \cap pMp$ for which there exists a unitary $u \in \mathcal U(M)$ such that $u(r_0Pr_0)u^* \subset L\Gamma \ootimes N$. One has to show that $r_0 = p$. Otherwise we can cut by $r = p - r_0$, and we obtain an algebra $rQ \subset rMr$ such that $rQ \prec_M L\Gamma \ootimes N$ and $rQ \nprec_M 1 \otimes N$. Remark that $rPr \subset \mathcal{QN}_{rMr}(rQ)''$. Applying (1), we get that there exists a non-zero partial isometry $v \in rM$, such that $vv^* \in (rPr)' \cap rMr \subset (rQ)' \cap rMr$ and $v^*(rPr)v \subset L\Gamma \ootimes N$.

Since $L\Gamma \ootimes N$ is a factor, modifying $v$ if necessary, one can assume that $v^*v \perp ur_0u^*$. Now the following ``cutting and pasting'' argument contradicts the maximality of $r_0$. The partial isometry $w_0 = ur_0 + v^*$ satisfies $w_0^*w_0 = r_0 + vv^* \in Q' \cap pMp$ and $w_0(r_0 + vv^*)Qw_0^* \subset L\Gamma$. Extending $w_0$ into a unitary, we obtain a $w \in \mathcal U(M)$ satisfying $w(r_0 + vv^*)Qw^* \subset L\Gamma$.
\end{proof}

\subsection{Gaussian actions}
\label{gaussianactions}

To any orthogonal representation $\pi: \Gamma \rightarrow \mathcal O(H)$ of a discrete countable group, one can associate a trace preserving action $\sigma_\pi: \Gamma \curvearrowright A$ on an abelian von Neumann algebra, called the Gaussian action associated with $\pi$. This Gaussian action can be constructed as follows. For more explicit constructions, see \cite[Appendix A.7]{BHV08} or \cite{PS10}.
Consider the unique abelian finite von Neumann algebra $(A,\tau)$ generated by unitaries $(w(\xi))_{\xi \in H}$ such that:
\begin{itemize}
\item $w(0) = 1$ and $w(\xi + \eta) = w(\xi)w(\eta)$, $w(\xi)^* = w(-\xi)$, for all $\xi,\eta \in H$;
\item $\tau(w(\xi)) = \exp(-\Vert \xi \Vert^2)$, for all $\xi \in H$.
\end{itemize}
It is easy to check that these conditions imply that the vectors $(w(\xi))_{\xi \in H_\R}$ are linearly independent and span a weakly dense $^*$-subalgebra of $A$. Then the Gaussian action $\sigma_\pi$ is defined by $(\sigma_\pi)_g(w(\xi)) = w(\pi(g)\xi)$, for all $g \in \Gamma, \xi \in H$.

As explained in \cite{Fu07} or \cite{PS10}, Gausssian actions are $s$-malleable in the sense of Popa \cite{Po08}: the rotation operators $\theta_t$, $t \in \R$ on $H \oplus H$ and the symmetry $\rho$ defined by \[\theta_t = \begin{pmatrix} \cos(\pi t/2) & -\sin(\pi t /2) \\ \sin(\pi t /2) & \cos(\pi t/2) \end{pmatrix} \text{ and } \rho = \begin{pmatrix} 1&0\\0&-1 \end{pmatrix}\]
give rise to a one-parameter group of automorphisms $\alpha_t$ and an automorphism $\beta$ of $A \ootimes A$, which commute with the diagonal action of $\Gamma$, and satisfy $\beta \circ \alpha_t = \alpha_{-t} \circ \beta$.

Now consider the von Neumann algebras $M = A \rtimes \Gamma$ and $\tilde{M} = (A \ootimes A) \rtimes_{\sigma \otimes \sigma} \Gamma$. View $M$ as a subalgebra of $\tilde{M}$ using the identification $M \simeq (A \ootimes 1)\rtimes \Gamma$.
The automorphisms defined above then extend to automorphisms of $\tilde{M}$ still denoted $(\alpha_t)$ and $\beta$, in such a way that $\alpha_t(u_g) = \beta(u_g) = u_g$, for all $g \in \Gamma$.

$(\alpha_t)$ is then easily seen to be an $s$-malleable deformation of the action, so it satisfies the so-called transversality property.

\begin{lem}[\cite{Po08}, Lemma 2.1]
\label{transversality}
For any $x \in M$ and $t \in \R$ one has $$\Vert x - \alpha_{2t}(x) \Vert_2 \leq 2\Vert \alpha_t(x) - E_M \circ \alpha_t(x)\Vert_2.$$
\end{lem}

With more conditions on the representation $\pi$, we also get the spectral gap property.

\begin{lem}[Spectral gap, \cite{Bo12}]
\label{spectralgap}
Assume that the representation $\pi$ is such that $\pi^{\otimes l}$ is weakly contained in the regular representation for some $l \geq 1$. Let $\omega \in \beta \N \setminus \N$ be a free ultrafilter on $\N$. 

Then for every subalgebra $Q \subset M$ with no amenable direct summand, one has $Q' \cap \tilde{M}^\omega \subset M^\omega$
\end{lem}

In fact this lemma admits a relative version.

Recall from \cite{OP07} that if $(M,\tau)$ is a finite von Neumann algebra, $p \in M$ a projection, and $Q \subset M$ and $P \subset pMp$ are subalgebras, one says that $P$ is {\it amenable relative to} $Q$ inside $M$ if there exists a $P$-central state $\varphi$ on $p\langle M,e_Q \rangle p$ such that $\varphi(pxp) = \tau(pxp)/\tau(p)$ for any $x \in M$. Here $\langle M,e_Q \rangle$ denotes Jones' basic construction associated with the inclusion $Q \subset M$.
Following \cite[Section 2.4]{IPV11}, $P$ is said to be {\it strongly non-amenable relative to} $Q$ if for all non-zero projection $p_1 \in P' \cap pMp$, $Pp_1$ is not amenable relative to $Q$.

\begin{lem}
\label{gspectralgap}
Assume that the representation $\pi$ is such that $\pi^{\otimes l}$ is weakly contained in the regular representation for some $l \geq 1$. Let $\omega \in \beta \N \setminus \N$ be a free ultrafilter on $\N$. 

Then for any subalgebra $Q \subset M \ootimes N$ which is strongly non-amenable relative to $1 \otimes N$, one has $Q' \cap (\tilde{M} \ootimes N)^\omega \subset Q' \cap (M\ootimes N)^\omega$.
\end{lem}

\subsection{Deformation/rigidity results for Gaussian actions}

We mention here different versions of statements that we proved in \cite{Bo12} using deformation/rigidity arguments.

The following result is a variation of \cite[Theorem 3.4]{Bo12}, with a formulation closer to \cite[Theorem 4.2]{IPV11}.

\begin{thm}
\label{adapted IPV}
Assume that $\Gamma \curvearrowright A$ is the Gaussian action associated with a mixing representation of an ICC group $\Gamma$. Let $N$ be a II$_1$ factor. Put $M = A \rtimes \Gamma$ and define $(\alpha_t)$ as in section \ref{gaussianactions}.

Let $p \in M \ootimes N$ and $Q \subset p(M \ootimes N)p$ be a subalgebra such that there exist $t_0 = 1/2^n$, $z \in M \ootimes N$ and $c > 0$ satisfying
\[\vert \tau((\alpha_{t_0} \otimes \id)(u^*)zu)\vert \geq c, \; \text{for all } u \in \mathcal U(Q).\]
Put $P = \mathcal{QN}_{p(M \ootimes N)p}(Q)''$. Then at least one of the following assertions holds.
\begin{enumerate}
\item $Q \prec 1 \otimes N$;
\item $P \prec A \ootimes N$;
\item There exists a non-zero partial isometry $v \in pM$ such that $vv^* \in  \mathcal Z(P)$ and $v^*Pv \subset L\Gamma \ootimes N$.
\end{enumerate}
\end{thm}
\begin{proof}
Assume that (2) is not satisfied. Using the fact that $\pi$ is mixing, the same proof that the one of \cite[Theorem 3.4]{Bo12} gives that $Q \prec L\Gamma \ootimes N$. Now if $Q \nprec 1 \otimes N$, Proposition \ref{embedgroup}(1) implies that (3) holds true.
\end{proof}

Now one can deduce the Gaussian version of \cite[Corollary 4.3]{IPV11}, which implies Step (1) of the proof of Theorems A and B.

\begin{cor}
\label{corIPV}
Assume that $\Gamma$ is ICC and let $\Gamma \curvearrowright A$ be a mixing Gaussian action. Put $M = A \rtimes \Gamma$. Let $N$ be a II$_1$ factor and $Q \subset p(M \ootimes N)p$ a subalgebra, for some $p \in M \ootimes N$. Assume that we are in one of the following situations:
\begin{itemize}
\item $Q \subset p(M \ootimes N)p$ has the relative property (T);
\item $Q' \cap p(M \ootimes N)p$ is strongly non-amenable relative to $1 \otimes N$ (see end of Section 2.2), and some tensor power of $\pi$ is weakly contained in the regular representation of $\Gamma$.
\end{itemize}
Denote by $P = \mathcal{QN}_{p(M \ootimes N)p}(Q)''$. Then one of the following assertions is true.
\begin{enumerate}
\item $Q \prec 1 \otimes N$;
\item $P \prec A \ootimes N$;
\item There exists a unitary $v \in M \ootimes N$ such that $v^*Pv \subset L\Gamma \ootimes N$.
\end{enumerate}
\end{cor}
\begin{proof}
The assumptions imply that the deformation $\alpha_t \otimes id$ converges to the identity uniformly on $(Q)_1 = \lbrace x \in Q, \Vert x \Vert \leq 1 \rbrace$ in $\Vert \cdot \Vert_2$. Indeed, if  $Q \subset p(M \ootimes N)p$ has relative property (T), this is almost by definition. If $Q' \cap p(M \ootimes N)p$ is strongly non-amenable relative to $1 \otimes N$, then this is a consequence of spectral gap lemma \ref{gspectralgap}, and transversality property \ref{transversality} (see the proof of \cite[Lemma 5.2]{Po08}).

Hence, for all $r \in Q' \cap pMp$, the subalgebra $rQ \subset rMr$ satisfies the assumpions of Theorem \ref{adapted IPV}. Then if $Q \nprec 1 \otimes N$ and $P \nprec A \ootimes N$, Theorem \ref{adapted IPV} applied to all such $rQ$'s implies in particular that for all $r \in Q' \cap pMp$, $rQ \prec L\Gamma \ootimes N$. Now (3) follows from Lemma \ref{embedgroup}(2).
\end{proof}

In \cite{Bo12}, we also obtained a localization result (Theorem 3.8) for subalgebras of $M$ that commute inside $M^\omega$ with rigid subalgebras of $M^\omega$, for some free ultrafilter $\omega$ on $\N$. In fact, the same proof leads to the following improvement. We include a sketch of the proof for convenience.

\begin{thm}
\label{ultraproduct}
Let $\Gamma \curvearrowright A$ be a mixing Gaussian action. Put $M = A \rtimes \Gamma$ and consider a II$_1$ factor $N$. Assume that $(v_n)$ is a bounded sequence of elements in $M \ootimes N$ such that $\alpha_t \otimes \id$ converges to the identity uniformly on the set $\lbrace v_n, n \in \N \rbrace$. Choose a free ultrafilter $\omega$ on $\N$, and denote by $D \subset M \ootimes N \subset (M \ootimes N)^\omega$ the subalgebra of elements that commute with the element $(v_n)_n \in (M \ootimes N)^\omega$. Put $P = \mathcal{QN}_{M \ootimes N}(D)''$.

Then one of the following is true.
\begin{enumerate}
\item $(v_n)_n \in (A \ootimes N)^\omega \rtimes \Gamma$;
\item $D \prec L\Gamma \ootimes N$;
\item $P \prec_M A \ootimes N$.
\end{enumerate}
\end{thm}
\begin{proof}[Sketch of proof]
Assume that $(v_n)_n \notin (A \ootimes N)^\omega \rtimes \Gamma$. We will show that the $D$ satisfies the assumptions of Theorem \ref{adapted IPV}.

Define $x = (x_n) = (v_n) - E_{(A \ootimes N)^\omega \rtimes \Gamma}((v_n)) \neq 0$. Dividing $x$ if necessary by $\Vert x \Vert_2$, one can assume that $\Vert x \Vert_2 \leq 1$. For $F \subset \Gamma$ finite, denote by $P_F: L^2(M) \rightarrow L^2(M)$ the projection onto the closed linear span of elements of the form $xu_g$, $x \in A$ $g \in F$. One checks that:
\begin{itemize}
\item $\alpha_t \otimes id$ converges to identity uniformly on $\lbrace x_n, n \in \N \rbrace$;
\item $\lim_n \Vert [x_n,u] \Vert_2 \rightarrow 0$,  for any $u \in \mathcal U(D)$;
\item $\lim_n \Vert (P_F \otimes id)(x_n)\Vert_2 \rightarrow 0$, for any finite subset $F \subset \Gamma$.
\end{itemize}
Using \cite[Lemma 3.8]{Va10b}, one can show that this last condition implies that \[\lim_n \langle x_n \xi x_n^*,\eta \rangle = 0, \forall \xi,\eta \in (L^2(\tilde{M}) \ominus L^2(M)) \otimes L^2(N).\]

Fix $\varepsilon > 0$. Then there exists a $t = 1/2^k$ such that $\Vert (\alpha_t \otimes id)(x_n) - x_n\Vert_2 < \varepsilon$, $\forall n$.

Fix $u \in \mathcal U(D)$ and put $\delta_t(u) = (\alpha_t\otimes id)(u) - E_{M \otimes N}((\alpha_t \otimes id)(u))$. Then $\delta_t(u) \in (L^2(\tilde{M}) \ominus L^2(M)) \otimes L^2(N)$, and we get
\begin{align*}
\lim_n \Vert \delta_t(u)x_n \Vert_2^2 & \approx_{2\varepsilon} \lim_n \langle \delta_t(ux_n),\delta_t(u)x_n \rangle\\
& \approx_{2\varepsilon} \lim_n \langle \delta_t(x_nu),\delta_t(u)x_n \rangle\\
& \approx_{4\varepsilon} \lim_n \langle x_n\delta_t(u)x_n^*,\delta_t(u) \rangle = 0.
\end{align*}
Thus $\lim_n \Vert \delta_t(u)x_n \Vert_2^2 \leq 4\varepsilon$. 
But exactly as in the proof of Popa's transversality lemma, one shows that for all $n \in \N$,
\begin{align*}
\Vert (\alpha_{2t} \otimes id)(u)x_n - ux_n \Vert_2 & \leq \Vert (\alpha_t \otimes id)(u) x_n - (\alpha_{-t} \otimes id)(u)x_n\Vert_2 + 2\varepsilon\\
& \leq 2 \Vert \delta_t(u)x_n\Vert_2 + 2\varepsilon.
\end{align*}
Hence, if $\varepsilon < 1$, we get that $\lim_n \Vert (\alpha_{2t} \otimes id)(u)x_n - ux_n \Vert_2 < 6\sqrt{\varepsilon}$.

Put $z = E_{M \ootimes N}((x_nx_n^*)_n)$. We have
\[2\lim_n \Vert x_n \Vert_2^2 - 2\Re(\tau((\alpha_{2t} \otimes id)(u^*)zu)) = \lim_n \Vert (\alpha_{2t} \otimes id)(u)x_n - ux_n \Vert_2^2 < 36\varepsilon.\]
If $\varepsilon$ was chosen to be small enough, this implies the result.
\end{proof}

\begin{cor}
\label{adaptedioana}
For $i = 1,2$, consider mixing Gaussian actions $\Gamma_i \curvearrowright A_i$ and  put $M_i = A_i \rtimes \Gamma_i$, $A = A_1 \ootimes A_2$, $\Gamma = \Gamma_1 \times \Gamma_2$ and $M = M_1 \ootimes M_2 = A \rtimes \Gamma$.

For $i = 1,2$, define $\tilde M_i$ and $(\alpha_t^i)$, as in section \ref{gaussianactions}, and denote by $\tilde M = \tilde{M}_1 \ootimes \tilde{M}_2$ equipped with the deformation $(\alpha_t) = (\alpha^1_t \otimes \alpha^2_t)$.

Assume that $(v_n)$ is a bounded sequence of elements in $M$ such that $\alpha_t$ converges uniformly to the identity on the set $\lbrace v_n, n \in \N \rbrace$. Choose a free ultrafilter $\omega$ on $\N$, and denote $D \subset M \subset M^\omega$ the subalgebra of elements that commute with the element $(v_n)_n \in M^\omega$. Put $P = \mathcal{QN}_M(D)''$.

Then one of the following is true.
\begin{enumerate}
\item $(v_n)_n \in A^\omega \rtimes \Gamma$;
\item $D \prec_M L\Gamma_1 \ootimes M_2$ or $D \prec_M M_1 \ootimes L\Gamma_2$;
\item $P \prec_M A_1 \ootimes M_2$ or $P \prec_M M_1 \ootimes A_2$.
\end{enumerate}
\end{cor}
\begin{proof}
Exactly as in the proof of \cite[Theorem 3.2]{Io11} Claim 2, we get that if $\alpha_t$ converges uniformly on $\lbrace v_n, n \in \N \rbrace$, then so do $\alpha_t^1 \otimes id$ and $\id \otimes \alpha_t^2$. Thus if (2) and (3) are not satisfied, Theorem \ref{ultraproduct} implies that $(v_n) \in \left( (A_1 \ootimes M_2)^\omega \rtimes \Gamma_1 \right) \cap \left( (M_1 \ootimes A_2)^\omega \rtimes \Gamma_2 \right) = A^\omega \rtimes \Gamma$.
\end{proof}

\subsection{$2$-mixing property}

\begin{defn}
A trace-preserving action $\Gamma \curvearrowright^\sigma A$ of a countable group on an abelian von Neumann algebra is said to be {\it $2$-mixing} if for any $a, b, c \in A$, the quantity $\tau(a\sigma_{g}(b)\sigma_{h}(c))$ tends to $\tau(a)\tau(b)\tau(c)$ as $g, h, g^{-1}h$ tend to infinity.
\end{defn}

\begin{prop}
An action $\Gamma \curvearrowright^\sigma A$ is $2$-mixing if and only if for all $a, b, c \in A$, one has
\[ \vert \tau(a\sigma_g(b)\sigma_h(c)) - \tau(a)\tau(\sigma_g(b)\sigma_h(c)) \vert \rightarrow 0,\]
when $g \rightarrow \infty$, $h \rightarrow \infty$.
\end{prop}
\begin{proof}
The {\it if} part is straightforward. For the converse, assume that $\sigma$ is $2$-mixing. It is sufficient to show that if $a,b,c \in A$, with $\tau(a) = 0$, then $\tau(a\sigma_g(b)\sigma_h(c)) \rightarrow 0$, as $g,h \rightarrow \infty$.

Assume by contradiction that there exist sequences $g_n, h_n \in \Gamma$ going to infinity, and $\delta > 0$ such that $\vert \tau(a\sigma_{g_n}(b)\sigma_{h_n}(c)) \vert \geq \delta$, for all $n$.
Then two cases are possible:

{\bf Case 1.} The sequence $g_n^{-1}h_n$ is contained in a finite set. Then taking a subsequence if necessary, one can assume that $g_n^{-1}h_n = k$ is constant. Then for all $n$, we get \[\tau(a\sigma_{g_n}(b)\sigma_{h_n}(c)) = \tau(a\sigma_{g_n}(b\sigma_k(c)).\]
But since $\sigma$ is mixing this quantity tends to $0$ as $n$ tends to infinity.

{\bf Case 2.} The sequence $g_n^{-1}h_n$ is not contained in a finite set. Then taking a subsequence if necessary, one can assume that $g_n^{-1}h_n \rightarrow \infty$ when $n \rightarrow \infty$. Then the $2$-mixing implies that $\tau(a\sigma_{g_n}(b)\sigma_{h_n}(c)) \rightarrow 0$.

In both cases, we get a contradiction.
\end{proof}

Of course any $2$-mixing action is mixing. The converse holds for Gaussian actions.

\begin{prop}
If $\Gamma \curvearrowright^\sigma A$ is the Gaussian action associated with a mixing representation $\pi$ on $H$, then $\sigma$ is $2$-mixing.
\end{prop}
\begin{proof}
By a linearity/density argument, it is enough to prove that for all $\xi,\eta,\delta \in H$, and all sequences $g_n,h_n \in \Gamma$ tending to infinity, one has
\[ \lim_n \left[ \tau(\omega(\xi) \sigma_{g_n}(\omega(\eta))\sigma_{h_n}(\omega(\delta)))- \tau(\omega(\xi))\tau(\sigma_{g_n}(\omega(\eta))\sigma_{h_n}(\omega(\delta))) \right] = 0.\]
But one checks that:
\begin{itemize}
\item $\tau(\omega(\xi) \sigma_{g_n}(\omega(\delta))\sigma_{h_n}(\omega(\delta)))  = \exp(-\Vert \xi + \pi(g_n)\eta + \pi(h_n)\delta \Vert^2)$;
\item $\tau(\omega(\xi))\tau(\sigma_{g_n}(\omega(\eta))\sigma_{h_n}(\omega(\delta))) = \exp(-\Vert \xi \Vert^2 - \Vert \pi(g_n)\eta + \pi(h_n)\delta \Vert^2)$.
\end{itemize}
The difference is easily seen to tend to $0$ as $n \rightarrow \infty$.
\end{proof}

\section{The key step}

We now state the key theorem from which Theorems A and B follow as explained in the introduction.

\begin{thm}
\label{keythm}
For $i = 1,2$, consider mixing Gaussian actions $\Gamma_i \curvearrowright A_i$ of discrete countable groups $\Gamma_i$, and put $M_i = A_i \rtimes \Gamma_i$, $A = A_1 \ootimes A_2$, $\Gamma = \Gamma_1 \times \Gamma_2$ and \[M = M_1 \ootimes M_2 = A \rtimes \Gamma.\]
Let $t > 0$. Realize $(L\Gamma)^t \subset M^t$ by fixing an integer $n \geq t$ and $p \in L\Gamma \otimes M_n(\C)$ with trace $t/n$. Let $D \subset M^t$ be an abelian subalgebra, and denote by $\Lambda = \mathcal{N}_{M^t}(D) \cap \mathcal U((L\Gamma)^t)$ and make the following assumptions:
\begin{enumerate}[(i)]
\item $\Lambda'' \nprec_M L\Gamma_1 \otimes 1$ and $\Lambda'' \nprec_M 1 \otimes L\Gamma_2$;
\item $D \nprec L\Gamma_1 \otimes M_2$ and $D \nprec_M M_1 \otimes L\Gamma_2$.
\end{enumerate}
Denote by $C = D' \cap M^t$. Then for all projections $q \in \mathcal Z(C)$, $Cq \prec_M A$.
\end{thm}

\begin{proof}
Exactly as in the proof of \cite[Theorem 5.1]{IPV11}, we first show that it is sufficient to prove that $C \prec_M A$.
Indeed, assume that we have shown that the assumptions of the theorem imply that $C \prec_M A$.

Consider the set of projections \[\mathcal P  = \lbrace q_1 \in \mathcal Z(C)\, \vert \, Cq \prec_M A, \text{ for all non-zero subprojections } q \in \mathcal{Z}(C)q_1 \rbrace.\]
Then $\mathcal P$ admits a unique maximal element $p_1 \in \mathcal Z(C)$. By uniqueness, $p_1$ commutes with the normalizer of $C$, and in particular with $\Lambda''$. Using \cite[Lemma 3.8]{Va10b} and assumption (i), we get that $p_1 \in (L\Gamma)^t$. We want to show that $p_1 = p (= 1_C)$. Otherwise, we can cut by $p-p_1$ and we see that $(p-p_1)D \subset (p-p_1)(M \otimes M_n(\C))(p-p_1)$ satisfies the assumptions of the theorem. Thus $(p-p_1)C \prec_M A$. This contradicts the maximality of $p_1$.

So the rest of the proof is devoted to showing that $C \prec_M A$. As in the proof of \cite[Theorem 5.1]{IPV11}, we assume that $t \leq 1$, so that $n$ can be chosen to be equal to $1$. This assumption largely simplifies notations, and does not hide any essential part of the proof.

Note that the assumption (i) implies that there exists a sequence of unitary elements $v_n \in \mathcal U(pL\Gamma p)$ that normalize $D$ and such that
\begin{equation}
\label{eq1}
\Vert E_{L\Gamma_1 \otimes 1}(av_nb)\Vert_2 \rightarrow 0 \text{ and } \Vert E_{1 \otimes L\Gamma_2}(av_nb)\Vert_2 \rightarrow 0, \, \forall a,b \in M.
\end{equation}
We will proceed in two steps to prove that $C \prec_M A$. In a first step we collect properties regarding the sequence $(v_n)$ or sequences of the form $(v_nav_n^*)$, $a \in D$. In the second step we show the result, reasoning by contradiction. Before moving on to these two steps, we introduce some notations:
\begin{itemize}
\item We denote by $u_g, g \in \Gamma$ the canonical unitaries in $M$ implementing the action of $\Gamma$;
\item For any element $x \in M$, we denote by $x = \sum_{g \in \Gamma} x_gu_g$ ($x_g \in A$ for all $g \in \Gamma$) its Fourier decomposition.
\item If $S \subset \Gamma$ is any subset, denote by $P_S: L^2(M) \rightarrow L^2(M)$ the projection onto the linear span of the vectors $au_g$, $a \in A$, $g \in S$.
\item If $K \subset A$ is a closed subspace, we denote by $Q_K: L^2(M) \rightarrow L^2(M)$ the projection onto the linear span of the vectors $au_g$, $a \in K$, $g \in \Gamma$.
\end{itemize}
Warning for the sequel: ``$g,h \in \Gamma$" means that $g$ and $h$ are two elements $(g_1,g_2)$ and $(h_1,h_2)$ of the product group $\Gamma = \Gamma_1 \times \Gamma_2$. It is different from ``$(g,h) \in \Gamma$"!

\subsection*{Step 1: Properties of the sequences $(v_nav_n^*)$, $a \in D$}

\begin{lem}
\label{lem1}
For any free ultrafilter $\omega$ on $\N$, and any $a \in D$, the element $(v_nav_n^*)_n \in M^\omega$ belongs to $A^\omega \rtimes \Gamma$.
\end{lem}
\begin{proof}
We will apply Corollary \ref{adaptedioana}. Fix $a \in D$. Since the $v_n$'s are in $L\Gamma$, the deformation $\alpha_t$ introduced in the statement of Corollary \ref{adaptedioana} converges uniformly on the set $\lbrace v_nav_n^*, n \in \N \rbrace$. Thus Corollary \ref{adaptedioana} implies that one of the following holds true:
\begin{itemize}
\item $(v_nav_n^*)_n \in A^\omega \rtimes \Gamma$;
\item $D \prec_M L\Gamma_1 \ootimes M_2$ or $D \prec_M M_1 \ootimes L\Gamma_2$;
\item $P \prec_M A_1 \ootimes M_2$ or $P \prec_M M_1 \ootimes A_2$, where $P = \mathcal N_{pMp}(D)''$.
\end{itemize}
The second case is excluded by assumption, so we are left to showing that the third case is not possible. By symmetry, it is sufficient to show that $P \nprec_M A_1 \ootimes M_2$. But we claim that for all $x, y \in M$, $\Vert E_{A_1 \ootimes M_2}(xv_ny)\Vert_2 \rightarrow 0$.
Since $v_n \in \mathcal{U}(P)$, this claim implies the result.

By Kaplansky's density theorem, and by linearity it is sufficient to prove the claim for $x$ and $y$ of the form $u_g \otimes 1$, $g \in \Gamma_1$. In particular $xv_ny$ lies in $L\Gamma$. So using the fact that \[\begin{matrix} A_1 \ootimes M_2 & \subset & M\\ \cup && \cup\\ 1 \otimes L\Gamma_2 & \subset & L\Gamma \end{matrix}\] is a commuting square, \ref{eq1} directly implies that $\Vert E_{A_1 \ootimes M_2}(xv_ny)\Vert_2 \rightarrow 0$.
\end{proof}

For an element $x \in M = L\Gamma$, denote by $h(x)$ the height of $x$: $h(x) = \sup_{g \in \Gamma} \vert x_g \vert$, where $x = \sum x_gu_g$ is the Fourier decomposition of $x$.

\begin{lem}
\label{lem2}
There exists $\delta > 0$ such that $h(v_n) > \delta$ for all $n$.
\end{lem}
\begin{proof}
Assume that the result is false. Taking a subsequence if necessary, we get that $h(v_n) \rightarrow 0$. Then we claim that for all finite subset $S \subset \Gamma$, and all $a \in (M \ominus (L\Gamma_1 \ootimes M_2)) \cap (M \ominus (M_1 \ootimes L\Gamma_2))$, \[\lim_n \Vert P_S(v_nav_n^*)\Vert_2 = 0.\]
Note that $(M \ominus (L\Gamma_1 \ootimes M_2)) \cap (M \ominus (M_1 \ootimes L\Gamma_2))$ is the subset of elements in $M$ whose Fourier coefficients lie in the weak closure of $(A_1 \ominus \C1) \otimes (A_2 \ominus \C1)$.

By a linearity/density argument, to prove this claim it is sufficient to show that for any sequence of unitaries $w_n \in \mathcal U(p L\Gamma p)$ and $a \otimes b \in (A_1 \ominus \C 1) \otimes (A_2 \ominus \C 1)$,
\[\Vert E_A(v_n (a\otimes b) w_n) \Vert_2 \rightarrow 0.\]
Write $v_n = \sum_{g \in \Gamma} v_{n,g}u_g$ and $w_n = \sum_{g \in \Gamma} w_{n,g}u_g$. We have
\[E_A(v_n (a \otimes b) w_n) = \sum_{g \in \Gamma} v_{n,g}w_{n,g^{-1}}\sigma_g(a \otimes b),\]
which leads to the formula:
\begin{equation}\label{eq2}
\Vert E_A(v_n (a \otimes b) w_n) \Vert_2^2 = \sum_{g,g' \in \Gamma} v_{n,g}w_{n,g^{-1}}\overline{v_{n,g'}w_{n,g'^{-1}}} \tau(\sigma_g(a \otimes b)\sigma_{g'}(a^* \otimes b^*)).
\end{equation}
Fix $\varepsilon > 0$. Since the action $\Gamma_i \curvearrowright A_i$ is mixing for $i = 1,2$, there exist finite sets $F_i \subset \Gamma_i$ such that $\vert \tau((a \otimes b)\sigma_{(s,t)}(a^* \otimes b^*))\vert = \vert \tau(a\sigma_s(a^*))\tau(b\sigma_t(b^*))\vert < \varepsilon$, if $(s,t) \notin F = F_1 \times F_2$. Now \ref{eq2} and Cauchy-Schwarz inequality imply
\begin{align*}
\Vert E_A(v_n (a \otimes b) w_n) \Vert_2^2
& \leq \sum_{g \in \Gamma} \sum_{g' \in gF} \vert v_{n,g}w_{n,g^{-1}}\overline{v_{n,g'}w_{n,g'^{-1}}} \tau(\sigma_g(a \otimes b)\sigma_{g'}(a^* \otimes b^*))\vert + \varepsilon.\\
&\leq \Vert a \Vert_2^2 \Vert b \Vert_2^2 h(v_n) \vert F \vert \sum_{g \in \Gamma} \vert v_{n,g}w_{n,g^{-1}}\vert + \varepsilon\\
& \leq \Vert a \Vert_2^2 \Vert b \Vert_2^2 h(v_n) \vert F \vert + \varepsilon.
\end{align*}
Hence, $\limsup_n \Vert E_{A \ootimes A}(v_n (a \otimes b) w_n) \Vert_2^2 \leq \varepsilon$. Since $\varepsilon$ was arbitrary, we get the claim.

Now take $\varepsilon' < \Vert p \Vert_2/4$. By assumption, $D \nprec_M L\Gamma_1 \ootimes M_2$ and $D \nprec_M M_1 \ootimes L\Gamma_2$, so there exists $a \in \mathcal{U}(D)$ such that \[\Vert E_{L\Gamma_1 \ootimes M_2}(a)\Vert_2 < \varepsilon' \text{ and } \Vert E_{M_1 \ootimes L\Gamma_2}(a)\Vert_2 < \varepsilon'.\]
By Lemma \ref{lem1}, the sequence $(v_nav_n^*)_n$ belongs to $A^\omega \rtimes \Gamma$, so that there exists a finite subset $F \subset \Gamma$ such that $\Vert P_F(v_nav_n^*) \Vert_2 \geq \Vert p \Vert_2 - \varepsilon'$. Thus if we Define $a_0 = a - E_{L\Gamma_1 \ootimes M_2}(a) - E_{M_1 \ootimes L\Gamma_2}(a - E_{L\Gamma_1 \ootimes M_2}(a))$, we get
\begin{align*}
\Vert p \Vert_2 - \varepsilon' & \leq \Vert P_F(v_nav_n^*) \Vert_2 \\
&\leq \Vert P_F(v_n(a - E_{L\Gamma_1 \ootimes M_2}(a))v_n^*) \Vert_2 + \varepsilon'\\
& \leq \Vert P_F(v_n a_0 v_n^*) \Vert_2 + 3\varepsilon'.
\end{align*}
But $a_0$ is orthogonal to $L\Gamma_1 \ootimes M_2$ and $M_1 \ootimes L\Gamma_2$, because the conditional expectations $E_{L\Gamma_1 \ootimes M_2}$ and $E_{M_1 \ootimes L\Gamma_2}$ commute. Therefore, when $n$ goes to infinity, the claim implies that $\Vert P_F(v_na_0v_n^*) \Vert_2 \rightarrow 0$ which leads to the absurd statement that $\Vert p \Vert_2 \leq 4\varepsilon' < \Vert p \Vert_2$.
\end{proof}

We end this paragraph by a lemma that localizes the Fourier coefficients of elements $v_nav_n^*$ inside $A$, for a particular (fixed) $a \in D$. In fact, this lemma will be the starting point of our reasoning by contradiction in Step 2 below, being the initialization of an induction process.

\begin{lem}
\label{lem3}
For a well chosen $a \in \mathcal U(D)$, there exists a $\delta_0 > 0$, a finite dimensional subspace $K \subset (A_1 \ominus \C1) \otimes (A_2 \ominus \C1)$, and a sequence $(g_n,h_n) \in \Gamma$ such that:
\begin{itemize}
\item $g_n,h_n \rightarrow \infty$, as $n \rightarrow \infty$;\item $\liminf \Vert Q_{\sigma_{(g_n,h_n)}(K)}(v_nav_n^*) \Vert_2 > \delta_0.$
\end{itemize}
\end{lem}
\begin{proof}
Put $\delta_1 = \liminf h(v_n) > 0$ and consider a sequence $(g_n,h_n) \in \Gamma$ such that $\vert v_{n,(g_n,h_n)} \vert = h(v_n)$ for all $n$. Now \ref{eq1} implies that the sequences $(g_n)$ and $(h_n)$ go to infinity with $n$. Moreover, we have
\[\limsup_n \Vert v_n - v_{n,(g_n,h_n)}u_{(g_n,h_n)} \Vert_2 = \sqrt{\Vert p \Vert_2^2 - \delta_1^2}.\]
Take $\varepsilon > 0$ such that $\sqrt{\Vert p \Vert_2^2 - \delta_1^2} + 4\varepsilon < \Vert p \Vert_2$. By assumption (ii), there exists $a \in \mathcal{U}(D)$ such that $\Vert E_{L\Gamma_1 \ootimes M_2}(a) \Vert_2 < \varepsilon$ and $\Vert E_{M_1 \ootimes L\Gamma_2}(a) \Vert_2 < \varepsilon$.
Thus the element $a_1 = a - E_{L\Gamma_1 \ootimes M_2}(a) - E_{M_1 \ootimes L\Gamma_2}(a - E_{L\Gamma_1 \ootimes M_2}(a))$ satisfies $\Vert a - a_1 \Vert_2 < 3\varepsilon$, and its Fourier coefficients are in $(A_1 \ominus \C1) \ootimes (A_2 \ominus \C1)$. We conclude that there exists a finite dimensional $K \subset (A_1 \ominus \C1) \otimes (A_2 \ominus \C1)$ such that, $\Vert a - Q_K(a) \Vert_2 < 4\varepsilon$.

Finally, we get that \[\Vert v_nav_n^* - v_{n,(g_n,h_n)}u_{(g_n,h_n)}Q_K(a)v_n^* \Vert_2 < \sqrt{\Vert p\Vert_2^2 - \delta_1^2} + 4\varepsilon.\]
Since $v_{n,(g_n,h_n)}u_{(g_n,h_n)}Q_K(a)v_n^*$ belongs to the image of the projection  $Q_{\sigma_{(g_n,h_n)}(K)}$, we get the result with $\delta_0 > 0$ defined by $\Vert p \Vert_2^2 - \delta_0^2 = (\sqrt{\Vert p \Vert_2^2 - \delta_1^2} + 4\varepsilon)^2$.
\end{proof}

\subsection*{Step 2: We show that $C \prec_M A$}

\begin{notation}
For a finite subset $G \subset \Gamma$, finite dimensional subspaces $K_1,K_2 \subset A$ and $\lambda > 0$, define 
\[ [K_1 \times \sigma_G(K_2)]^\lambda = \conv \lbrace \lambda a \sigma_g(b) \, \vert \, a \in K_1, b \in K_2, g \in G, \Vert a \Vert_2 \leq 1, \Vert b \Vert_2 \leq 1 \rbrace.\]
\end{notation}

We have that $[K_1 \times \sigma_G(K_2)]^\lambda$ is a closed convex subset $\mathcal C$ of $A$ (being the convex hull of a compact subset in a finite dimensional vector space). Then the set \[\tilde{\mathcal C} = \lbrace \sum_{g \in \Gamma} \xi_g \otimes \delta_g \in L^2(M) \vert \, \forall g \in \Gamma, \xi_g \in \mathcal C \rbrace\] is a closed convex subset of $L^2(M)$. Hence one can define the ``orthogonal projection onto this set" $Q_{\mathcal C}: L^2(M) \rightarrow L^2(M)$ as follows. For $x \in L^2(M)$, $Q_{\mathcal C}(x)$ is the unique point of $\tilde{\mathcal{C}}$ such that \[\Vert x - Q_{\mathcal C}(x)\Vert = \inf_{y \in \tilde{\mathcal C}} \Vert x-y \Vert. \]

\begin{rem}
This notation is consistent with the previous notation $Q_K$: If $K \subset A$ is a finite dimensional subspace, then $Q_K(a) = Q_{\mathcal C}(a)$, where $\mathcal C = [\C1 \times \sigma_{\lbrace e \rbrace}(K)]^\lambda$ as soon as $\lambda \geq \Vert a \Vert_2$.
\end{rem}

Before getting into the heart of the proof, we check some easy properties of these convex sets.

\begin{lem}
\label{lem4}
Fix $\lambda > 0$ and finite dimensional subspaces $K_1,K_2 \subset A$. Then there exists a constant $\kappa > 0$ such that for all finite $G \subset \Gamma$, and all $x \in [K_1 \times \sigma_G(K_2)]^\lambda$, \[ \Vert x \Vert_\infty \leq \kappa.\]
\end{lem}
\begin{proof}
Since $K_1$ and $K_2$ are finite dimensional, there exists a constant $c > 0$ such that $\Vert a \Vert_\infty \leq c\Vert a \Vert_2$ for all $a \in K_1$ or $a \in K_2$. One sees that $\kappa = \lambda c^2$ satisfies the conclusion of the lemma.
\end{proof}

\begin{lem}
\label{lem5}
For finite subsets $F, G \subset \Gamma$, and finite dimensional subspaces $K_1, K_2, K_1', K_2' \subset A$ and $\lambda,\lambda' > 0$, we have
\[ [K_1 \times \sigma_F(K_2)]^{\lambda} + [K_1' \times \sigma_G(K_2')]^{\lambda'} \subset [(K_1+K_1') \times \sigma_{G \cup F}(K_2+K_2')]^{\lambda + \lambda'}.\]
\end{lem}
\begin{proof} This is straightforward.
\end{proof}

{\it From now on, we assume by contradiction that $C \nprec_M A$.}
The contradiction we are looking for is then a direct consequence of the following implication.  Indeed, using Lemma \ref{lem3}, and iterating the implication enough times, we get the absurd statement that there exist unitaries $a_n = v_nav_n^*$ and elements $b_n$ of the form $Q_{\mathcal C_n}(a_n)$ such that $\liminf_n \Vert a_n - b_n \Vert_2^2$ is negative.

\begin{implication}
Fix $a \in \mathcal U(D)$ and put $a_n = v_nav_n^*$ for all $n$. Assume that there exists a sequence of finite subsets $F_n\times G_n \subset \Gamma = \Gamma_1 \times \Gamma_2$, finite dimensional subspaces $K_1 \subset A$, $K_2 \subset (A_1 \ominus \C1) \otimes (A_2 \ominus \C1)$, $\lambda > 0$ and $\delta > 0$ such that:
\begin{itemize}
\item $\sup_n \vert F_n \vert \vert G_n \vert < \infty$;
\item $F_n \rightarrow \infty$, $G_n \rightarrow \infty$;
\item $\limsup_n \Vert a_n - Q_{\mathcal C_n}(a_n)\Vert_2^2 < \Vert p \Vert_2^2 - \delta^2$, where $\mathcal C_n = [K_1 \times \sigma_{F_n\times G_n}(K_2)]^\lambda$.
\end{itemize}
Then there exists a sequence of finite subsets $F_n'\times G_n' \subset \Gamma$, finite dimensional subspaces $K_1' \subset A$, $K_2' \subset (A_1 \ominus \C1) \otimes (A_2 \ominus \C1)$, and $\lambda' > 0$ such that:
\begin{itemize}
\item $\sup_n \vert F_n' \vert \vert G_n' \vert < \infty$;
\item $F_n' \rightarrow \infty$, $G_n' \rightarrow \infty$;
\item $\limsup_n \Vert a_n - Q_{\mathcal C_n'}(a_n)\Vert_2^2 < \Vert p \Vert_2^2 - 3\delta^2/2$, where $\mathcal C_n' = [K_1' \times \sigma_{F_n' \times G_n'}(K_2')]^{\lambda'}$.
\end{itemize}
\end{implication}

The multiple mixing property will be used in the proof of this implication through the following lemma.

\begin{lem}
\label{lem6}
Let $x,y,z \in A_1 \otimes A_2$. For any sequences $g_n = (g_n^1,g_n^2)\in \Gamma$ and $h_n = (h_n^1,h_n^2) \in \Gamma$ such that $g_n^1,g_n^2,h_n^1,h_n^2 \rightarrow \infty$, we have \[\vert \tau(x\sigma_{g_n}(y)\sigma_{h_n}(z))- \tau(x)\tau(\sigma_{g_n}(y)\sigma_{h_n}(z)) \vert \rightarrow 0.\]
\end{lem}
\begin{proof}
Without loss of generality, one can assume that $x = x_1 \otimes x_2$, $y = y_1 \otimes y_2$, $z = z_1 \otimes z_2$. We have
\begin{itemize}
\item $\tau(x\sigma_{g_n}(y)\sigma_{h_n}(z)) = \tau(x_1\sigma_{g_n^1}(y_1)\sigma_{h_n^1}(z_1))\tau(x_2\sigma_{g_n^2}(y_2)\sigma_{h_n^2}(z_2))$;
\item $\tau(x)\tau(\sigma_{g_n}(y)\sigma_{h_n}(z)) = \tau(x_1)\tau(\sigma_{g_n^1}(y_1)\sigma_{h_n^1}(z_1))\tau(x_2)\tau(\sigma_{g_n^2}(y_2)\sigma_{h_n^2}(z_2))$.
\end{itemize}
So the result follows directly from the multiple mixing property of the Gaussian actions $\Gamma_i \curvearrowright A_i$, $i=1,2$.
\end{proof}

\begin{proof}[Proof of the implication]
Let $a$, $F_n$, $G_n$, $K_1$, $K_2$, $\lambda$, $\delta$ and $\mathcal C_n$ be as in the implication. Fix $\varepsilon > 0$, with $\varepsilon \ll \delta$. By Lemma \ref{lem1} one can find $S \subset \Gamma$ finite such that $\Vert a_n - P_S(a_n)\Vert_2 \leq \varepsilon$, for all $n$.
Hence we get that $\limsup_n \Vert a_n - P_S \circ Q_{\mathcal C_n}(a_n)\Vert_2 < \sqrt{\Vert p \Vert_2^2-\delta^2}+\varepsilon$.

Now following Ioana's idea, we will consider an element $d \in \mathcal U(C)$ with sufficiently spread out Fourier coefficients so that for $n$ large enough, $dP_S \circ Q_{\mathcal C_n}(a_n)d^*$ is almost orthogonal to $P_S \circ Q_{\mathcal C_n}(a_n)$, while it is still close to $a_n$. Then the sum $dP_S \circ Q_{\mathcal C_n}(a_n)d^* + P_S \circ Q_{\mathcal C_n}(a_n)$ should be even closer to $a_n$.

Let $\alpha > 0$ be a (finite) constant such that $\Vert x \Vert_\infty \leq \alpha \Vert x \Vert_2$, for all $x \in K_1$. Since $K_2 \subset (A_1 \ominus \C1) \otimes (A_2 \ominus \C1)$ is finite dimensional, the set \[L = \lbrace g \in \Gamma \, , \, \exists a,b \in K_2, \Vert a\Vert_2 \leq 1, \Vert b \Vert_2 \leq 1: \vert\langle \sigma_g(a),b\rangle \vert \geq \varepsilon/\vert S\vert^2\lambda^2\alpha^2 \rbrace\] is finite.
Hence for all $n$, $L_n = \cup_{g,h \in F_n\times G_n} gLh^{-1}$ is finite, with cardinal smaller or equal to $\vert F_n \vert^2 \vert G_n \vert^2 \vert L \vert$, which is itself majorized by some $N$, not depending on $n$.

Since $C \nprec A$, Ioana's intertwining criterion (Lemma \ref{ioanascriterion}) implies that there exists $d \in \mathcal{U}(C)$ such that $\Vert P_F(d) \Vert_2 \leq \varepsilon/\kappa\vert S \vert$, whenever $\vert F \vert \leq N$, where $\kappa$ is given by Lemma \ref{lem4} applied to $K_1$, $K_2$ and $\lambda$.

By Kaplansky's density theorem, one can find $d_0, d_1 \in M$, and $T = T_1 \times T_2 \subset \Gamma$ finite such that:
\begin{itemize}
\item $d_i = P_T(d_i)$, $i = 0,1$;
\item $\Vert d_0 - d \Vert_2 \leq \min(\varepsilon,\varepsilon/\kappa\vert S \vert)$, $\Vert d_1 - d^* \Vert_2 \leq \varepsilon$;
\item $\Vert d_i \Vert_\infty \leq 1$, $i = 0,1$.
\end{itemize}
Since $a_n \in D$ for all $n$ and $d \in C = D' \cap M$, we have $da_nd^* = a_n$. Thus for all $n$, $\Vert a_n - d_0a_nd_1 \Vert_2 \leq 2\varepsilon$, and so \[ \limsup_n \Vert a_n - d_0P_S \circ Q_{\mathcal C_n}(a_n)d_1 \Vert_2 \leq \sqrt{\Vert p \Vert_2^2-\delta^2}+3\varepsilon.\]

Now, for all $n$, put $T_n = T \setminus L_n$. By definition of $d$, $\Vert P_T(d) - P_{T_n}(d) \Vert_2 \leq \varepsilon/\kappa\vert S \vert$, hence $\Vert d_0 - P_{T_n}(d_0) \Vert_2 \leq 3\varepsilon/\kappa\vert S \vert$. Notice that $\Vert P_S \circ Q_{\mathcal C_n}(a_n)\Vert_\infty \leq \kappa\vert S \vert$, which implies that \[\limsup_n \Vert a_n - P_{T_n}(d_0)P_S \circ Q_{\mathcal C_n}(a_n)d_1 \Vert_2 \leq \sqrt{\Vert p \Vert_2^2-\delta^2}+6\varepsilon.\]
Denote by $x_n = P_S \circ Q_{\mathcal C_n}(a_n)$ and $y_n =  P_{T_n}(d_0)P_S \circ Q_{\mathcal C_n}(a_n)d_1$.

We want to show that $\limsup_n \vert \langle x_n,y_n \rangle \vert$ is small.

Write $d_0 = \sum_{g \in T}d_{0,g}u_g$, $a_n = \sum_h a_{n,h}u_h$, and $d_1 = \sum_{k \in T} d_{1,k}u_k$. We get
\begin{align*}
\langle y_n , x_n \rangle & = \sum_{\substack{g \in T_n, h \in S, k \in T\\ ghk \in S}} \tau(d_{0,g}\sigma_{gh}(d_{1,k})\sigma_g(Q_{\mathcal C_n}(a_{n,h}))Q_{\mathcal C_n}(a_{n,ghk})^*)\\
& = \sum_{\substack{g \in T, h \in S, k \in T\\ ghk \in S}} \mathbf{1}_{\lbrace g \in T_n \rbrace}\tau(d_{0,g}\sigma_{gh}(d_{1,k})\sigma_g(Q_{\mathcal C_n}(a_{n,h}))Q_{\mathcal C_n}(a_{n,ghk})^*).
\end{align*}

{\it Claim.} For all fixed $x,y \in A$, and $g \in T$, there exists $n_0$ such that for all $n \geq n_0$, and all $a,b \in \mathcal C_n$, \[\vert \mathbf{1}_{\lbrace g \in T_n \rbrace}\tau(xy\sigma_g(a)b^*) \rangle \vert \leq 2\varepsilon \Vert x \Vert_2 \Vert y \Vert_2/\vert S \vert^2.\]
To prove this claim, first recall that for all $n$, $\mathcal C_n = [K_1 \times \sigma_{F_n \times G_n}(K_2)]^\lambda$. Denote by $\tilde K_1 = \spa \lbrace xy\sigma_g(a)b^*, a,b \in K_1 \rbrace$. Since $\tilde K_1$ and $K_2$ have finite dimension and since $F_n,G_n \rightarrow \infty$, Lemma \ref{lem6} implies that there exists $n_0$ such that for $n\geq n_0$, and for all $s,t \in F_n\times G_n$ one has
\begin{equation}\label{eq3} \sup_{\substack{a \in \tilde{K}_1, \Vert a \Vert_2 \leq 1\\ b,c \in K_2, \Vert b \Vert_2 \leq 1, \Vert c \Vert_2 \leq 1}} \vert \tau(a\sigma_{gs}(b)\sigma_t(c^*)) - \tau(a)\tau(\sigma_{gs}(b)\sigma_t(c^*))\vert \leq \varepsilon \Vert x \Vert_2 \Vert y \Vert_2/\vert S \vert^2\lambda^2.
\end{equation}
Thus take $n \geq n_0$. By definition of $\mathcal C_n$, it is sufficient to prove that for all $a,b \in K_1$, $c,d \in K_2$, with $\Vert a \Vert_2, \Vert b \Vert_2, \Vert c \Vert_2, \Vert d \Vert_2 \leq 1$, and all $s,t \in F_n\times G_n$, \[\vert \mathbf{1}_{\lbrace g \in T_n \rbrace}\tau(xy\sigma_g(\lambda a\sigma_s(c))\lambda b^*\sigma_t(d^*)) \vert \leq 2\varepsilon\Vert x \Vert_2 \Vert y \Vert_2/\vert S \vert^2.\]
We can assume that $g \in T_n$. An easy calculation gives
\begin{align*}
\vert \tau(xy\sigma_g(\lambda a\sigma_s(c))\lambda b^*\sigma_t(d^*)) \vert & \leq \varepsilon\Vert x \Vert_2 \Vert y \Vert_2/\vert S \vert^2 + \lambda^2\vert \tau(xy\sigma_g(a)b^*)\tau(\sigma_{gs}(c)\sigma_t(d^*))\vert\\
& \leq \varepsilon\Vert x \Vert_2 \Vert y \Vert_2/\vert S \vert^2 + \lambda^2\Vert x \Vert_2\Vert y \Vert_2 \Vert a \Vert_\infty \Vert b \Vert_\infty \varepsilon/\vert S \vert^2\lambda^2\alpha^2\\
& \leq 2\varepsilon\Vert x \Vert_2 \Vert y \Vert_2/\vert S \vert^2,
\end{align*}
where the first inequality is deduced from \ref{eq3}, while the second is because $g \notin L_n$. So the claim is proven.

Now we can estimate $\vert \langle x_n,y_n \rangle \vert$, for $n$ large enough.
\begin{align*}
\vert \langle x_n,y_n \rangle \vert & \leq 
\sum_{g \in T, h \in S, k' \in S} \vert \mathbf{1}_{\lbrace g \in T_n \rbrace}\tau(d_{0,g}\sigma_{gh}(d_{1,h^{-1}g^{-1}k'})\sigma_g(Q_{C_n}(a_{n,h}))Q_{C_n}(a_{n,k'})^*)\vert\\
& \leq \sum_{g \in T, h \in S, k' \in S} 2\varepsilon \Vert d_{0,g} \Vert_2 \Vert d_{1,h^{-1}g^{-1}k'} \Vert_2/\vert S \vert^2\\
& \leq 2\varepsilon \Vert d_0 \Vert_2 \Vert d_1 \Vert_2 \leq 2 \varepsilon.
\end{align*}
Therefore, we obtain: 
\begin{itemize}
\item $\limsup_n \Vert a_n - x_n \Vert_2 < \sqrt{\Vert p \Vert_2^2 - \delta^2} + \varepsilon$;
\item $\limsup_n \Vert a_n - y_n \Vert_2 < \sqrt{\Vert p \Vert_2^2 - \delta^2} + 6\varepsilon$;
\item $\limsup_n \vert \langle x_n , y_n \rangle \vert \leq 2\varepsilon$.
\end{itemize}
Thus using the formula \[\Vert x - (y+z) \Vert_2^2 = \Vert x-y \Vert_2^2 + \Vert x - z \Vert_2^2 - \Vert x \Vert_2^2 + 2\Re \langle y,z \rangle,\]
one checks that $\limsup_n \Vert a_n - (x_n + y_n) \Vert_2^2 \leq \Vert p \Vert_2^2 - 3\delta^2/2$, if $\varepsilon$ is small enough.

Now observe that \[y_n = \sum_{g \in T_n,h \in S, k \in T} d_{0,g}\sigma_{gh}(d_{1,k})\sigma_g(Q_{\mathcal C_n}(a_{n,h}))u_{ghk}.\]
So let us check that $y_n$ has its Fourier coefficients in $[K_0 \times \sigma_{(T_1F_n)\times (T_2G_n)}(K_2)]^{\lambda\vert S \vert \vert T \vert}$, where $K_0 = \spa \lbrace d_{0,g}\sigma_{gh}(d_{1,k})\sigma_g(c), c \in K_1, g,k \in T, h \in S \rbrace$.

Fix $n \in \N$, and $s \in \Gamma$. Denote by $y_{n,s} = E_A(y_nu_s^*)$. We have \[y_{n,s} = \sum_{\substack{g \in T_n,h \in S, k \in T\\ ghk=s}} d_{0,g}\sigma_{gh}(d_{1,k})\sigma_g(Q_{\mathcal C_n}(a_{n,h})).\]
Thus it is a convex combination of terms of the form
\begin{align*}
\mathcal T & = \sum_{\substack{g \in T,h \in S, k \in T\\ ghk=s}} d_{0,g}\sigma_{gh}(d_{1,k})\sigma_g(\lambda a_h\sigma_{t_h}(b_h))\\ 
& = \frac{1}{\vert S \vert \vert T \vert}\sum_{\substack{g \in T,h \in S, k \in T\\ ghk=s}} \vert S \vert \vert T \vert d_{0,g}\sigma_{gh}(d_{1,k})\sigma_g(\lambda a_h\sigma_{t_h}(b_h)) ,
\end{align*}
for elements $a_h \in K_1$, $b_h \in K_2$, with $\Vert a_h \Vert_2,\Vert b_h \Vert_2 \leq 1$ and $t_h \in F_n\times G_n$, for all $h \in S$. But such terms $\mathcal T$ are themselves convex combinations of elements of the form $\lambda \vert S \vert \vert T \vert x\sigma_{gt}(y)$, with $x \in K_0, y \in K_2$, $\Vert x \Vert_2, \Vert y \Vert_2 \leq 1$ and $gt \in T(F_n\times G_n) = (T_1F_n) \times (T_2G_n)$.

Therefore, as pointed out in Lemma \ref{lem5}, $x_n + y_n$ has Fourier coefficients in $\mathcal C_n' = [K_1' \times \sigma_{F_n' \times G_n'}(K_2')]^{\lambda'}$, with $K_1' = K_1 + K_0$, $K_2' = K_2$, $\lambda' = \lambda + \lambda\vert S \vert \vert T \vert$, and $F_n' = F_n \cup T_1F_n$, $G_n' = G_n \cup T_2G_n$.

We conclude that:
\[\Vert a_n - Q_{C_n'}(a_n)\Vert_2^2 \leq \Vert p \Vert_2^2 - 3\delta^2/2,\]
which proves the implication.
\end{proof}
The proof of Theorem \ref{keythm} is complete.
\end{proof}

Taking $\Gamma_2 = \lbrace e \rbrace$ and $A_2 = \C$ we obtain a similar statement for a single mixing action $\Gamma \curvearrowright A$.

\begin{cor}
\label{corkey}
Assume that $\Gamma \curvearrowright A$ is a mixing Gaussian action. Denote by $M = A \rtimes \Gamma$. Consider an abelian subalgebra $D \subset pMp$, $p \in L\Gamma$, which is normalized by a sequence of unitaries $(v_n) \in \mathcal U(pL\Gamma p)$ with $v_n \rightarrow 0$ weakly. Put $C = D' \cap pMp$. Then one of the following is true:
\begin{itemize}
\item $D \prec_M L\Gamma$ 
\item For all $q \in \mathcal Z(C)$, $qC \prec_M A$.
\end{itemize}
\end{cor}

In fact, S. Vaes asked during his series of lectures at the IHP in Paris (spring 2011) whether such a corollary could hold for any mixing action. A. Ioana showed that this is true for Bernoulli shifts \cite[Theorem 6.2]{Io11}, and as we just showed, the proof can be adapted to Gaussian actions. In our proof, we only used the following properties of Gaussian actions:
\begin{itemize}
\item The 2-mixing property;
\item The malleability property.
\end{itemize}
Moreover, the malleability of Gaussian actions is only used to prove Lemma \ref{lem1} (\emph{i.e.} to show that the sequences $(v_nav_n^*)$, $a \in D$ lie in $A^\omega \rtimes \Gamma$). We suspect that this lemma might be shown only using multiple mixing properties, but we were not able to reach this conclusion.

We end this section by mentioning a generalization of Theorems A and B that considers {\it some} amplifications. The proof is the same, and still works because Popa's orbit equivalence superrigidity theorems (\cite[Theorem 5.2 and Theorem 5.6]{IPV11} and \cite[Theorem 1.3]{Po08}) handle such amplifications.

\begin{thm}
\label{mainthm}
Let $\Gamma$ be an ICC countable discrete group, and $\pi : \Gamma \rightarrow \mathcal O(H_\R)$ an orthogonal representation of $\Gamma$.
Make one of the following two assumptions:
\begin{itemize}
\item $\Gamma$ is w-rigid and ICC, and $\pi$ is mixing; 
\item $\Gamma$ is an ICC non-amenable product of two infinite groups and $\pi$ is mixing and admits a tensor power which is weakly contained in the regular representation.
\end{itemize}
Let $\Gamma \curvearrowright A$ be the Gaussian action associated with $\pi$ and put $M = A \rtimes \Gamma$.
Let $\Lambda \curvearrowright B$ be another free ergodic action on an abelian von Neumann algebra, and put $N = B \rtimes \Lambda$.

If for some $t \geq 1$,  $M \simeq N^t$, then $t = 1$, $\Gamma \simeq \Lambda$ and the actions $\Gamma \curvearrowright A$ and $\lambda \curvearrowright B$ are conjugate.
\end{thm}

\section{An application to group von Neumann algebras}

As another application of Theorem \ref{keythm}, we construct a large class of II$_1$ factors which are not stably isomorphic to group von Neumann algebras. These factors are the crossed-product von Neumann algebras of Gaussian actions associated with representations $\pi$ as in Theorem A or Theorem B, with the extra-assumption that $\pi$ is not weakly contained in the regular representation.

In \cite[Proposition 2.8]{Bo12}, such Gaussian actions were shown not to be conjugate to generalized Bernoulli shifts. Using Theorems A and B, we get that the associated factors are not isomorphic to crossed-product factors of Bernoulli actions, and in particular, to von Neumann algebras of certain wreath-product groups. However, showing that such factors are not isomorphic to algebras $L\Lambda$, with no assumptions on the group $\Lambda$ is much harder, and will require the work of Ioana, Popa and Vaes \cite{IPV11}.

\begin{thm}
Let $\Gamma$ be an ICC group and $\pi: \Gamma \rightarrow \mathcal O(H)$ a mixing orthogonal representation of $\Gamma$ such that one of the following two conditions holds.
\begin{itemize}
\item $\Gamma$ is w-rigid or
\item $\Gamma$ is non-amenable and is isomorphic to the product of two infinite groups, and some tensor power of $\pi$ is weakly contained in the regular representation of $\Gamma$.
\end{itemize}
Assume moreover that $\pi$ itself is not weakly contained in the regular representation. Let $\Gamma \curvearrowright^\sigma A$ be the Gaussian action associated with $\pi$ and put $M = A \rtimes \Gamma$.
Then $M$ is not stably isomorphic to a group von Neumann algebra.
\end{thm}
\begin{proof}
Let $\pi$ be an orthogonal representation as in the statement of the theorem. Assume by contradiction that there exists a countable group $\Lambda$ such that $M \simeq (L\Lambda)^t$ for some $t > 0$. Then adapting the proof of \cite[Theorem 8.2]{IPV11}, we get that $t = 1$, and $\Lambda \simeq \Sigma \rtimes \Gamma$, for some infinite abelian group $\Sigma$ and some action $\Gamma \curvearrowright \Sigma$ by automorphisms. Moreover, the initial Gaussian action $\sigma$ is conjugate to the action of $\Gamma$ on $L\Sigma$.

Now, since $\sigma$ is mixing, the action $\Gamma \curvearrowright \Sigma \setminus \lbrace e \rbrace$ has finite stabilizers. But then the Koopmann representation $\Gamma \rightarrow \mathcal U(\ell^2(\Sigma \setminus \lbrace e \rbrace))$ is weakly contained in the left regular representation. Thus, Proposition 1.7 in \cite{PS10} implies that $\pi$ is weakly contained in the regular representation, which contradicts our assumptions on $\pi$.
\end{proof}

By \cite[Proposition 2.9]{Bo12}, we know that for each $n \geq 3$, $\PSL(n,\Z)$ admits a representation as in Theorem C. Thus we obtain the existence of a II$_1$ factor $M_n$, which is not stably isomorphic to a group von Neumann algebra. But using Theorem \ref{mainthm}, we get that the $M_n$'s are pairwise non-stably isomorphic :
$ M_n \ncong (M_m)^t, \, \forall t > 0, \, \forall n \neq m$.

\bibliographystyle{alpha1}

\end{document}